\documentclass[11 pt]{amsart}

\usepackage[latin1]{inputenc}
\usepackage{amsmath,amsthm,amssymb,amscd,color, xcolor,mathtools,url}
\usepackage{bbm}
\usepackage{graphicx}

\newtheorem{thm}{Theorem}[section]

 \theoremstyle{definition}
 
 \theoremstyle{remark}

 \numberwithin{equation}{section}

\def\be#1 {\begin{equation} \label{#1}}
\def\ee{\end{equation}}

\def\sqw{\hbox{\rlap{\leavevmode\raise.3ex\hbox{$\sqcap$}}$%
\sqcup$}}
\def\findem{\ifmmode\sqw\else{\ifhmode\unskip\fi\nobreak\hfil
\penalty50\hskip1em\null\nobreak\hfil\sqw
\parfillskip=0pt\finalhyphendemerits=0\endgraf}\fi}

\title{$L^2$ to $L^p$ bounds for spectral projectors on thin intervals in Riemannian manifolds}

\subjclass[2000]{...}

\keywords{Riemannian manifolds, $L^p$ estimates, spectral projectors, eigenfunctions, quasimodes}

\begin{document}

\author[P. Germain]{Pierre Germain}
\address{Pierre Germain, Department of Mathematics, Huxley Building, South Kensington Campus,
Imperial College London, London SW7 2AZ, United Kingdom}
\email{pgermain@ic.ac.uk}

\begin{abstract}
Given a Riemannian manifold with Laplace-Beltrami operator $\Delta$, consider the projector on a thin interval $[\lambda-\delta,\lambda+\delta]$ associated to the self-adjoint operator $\sqrt{-\Delta}$. Viewed as an operator from $L^2(M)$ to $L^p(M)$, what is its operator norm? For $\delta=1$, this question is fully answered by a celebrated theorem of Sogge, which applies to any manifold. For $\delta < 1$, the global geometry of the manifold comes into play, and connections to a number of mathematical fields (such as Differential Geometry, Combinatorics, Number Theory) appear, but the problem remains mostly open. The aim of this article is to review known results, focusing on cases exhibiting symmetry (where the regime where $\delta$ is polynomially small in $\lambda$ becomes approachable) and emphasizing harmonic analytic rather than microlocal methods.
\end{abstract}

\maketitle

\tableofcontents

\section{Introduction}

Consider a (smooth, complete) Riemannian manifold $M$, of dimension $d$, with Laplace-Beltrami operator $\Delta$. Define the spectral projector $P_{\lambda,\delta}$ through functional calculus by
$$
P_{\lambda,\delta} = \mathbf{1}_{[\lambda-\delta,\lambda+\delta]} \left( \sqrt{-\Delta} \right).
$$
It is the orthogonal projector on the (generalized) eigenmodes of $\sqrt{-\Delta}$ with eigenvalue in $[\lambda-\delta,\lambda+\delta]$.
We ask the following question: 

\begin{center}
\textit{What is the operator norm of $P_{\lambda,\delta}$ from $L^2(M)$ to $L^p(M)$?}
\end{center}

\medskip

A more classical question is to ask for $L^p$ bounds of $L^2$-normalized eigenfunctions on compact manifolds, which would correspond to $\delta =0$. Similarly, the question above could be formulated in terms of $L^p$ bounds for $L^2$-normalized~{\it almost-}eigenfunctions.

Why is it meaningful to consider the case $\delta > 0$? First, this has the advantage of applying to any manifold, be it compact or not, with discrete or continuous spectrum, and this more general viewpoint is helpful. Second, in the case of compact manifolds, the case $\delta =0$ seems in general completely out of reach, while the case $\delta =1$ is fully understood, as we will see; thus, interpolating between these two situations seems natural. Finally, considering a window of positive width encodes information on the distribution of the spectrum, together with the size of the eigenfunctions.

\medskip

Before reviewing the theory, let us observe that the question of estimating $\| P_{\lambda,\delta} \|_{L^2 \to L^p}$ can be asked in different ways. First, an immediate $TT^*$ argument shows that
\begin{equation}
\label{TTstar}
\| P_{\lambda,\delta} \|_{L^2 \to L^p}^2 = \| P_{\lambda,\delta} \|_{L^{p'} \to L^p}.
\end{equation}
Second, it is essentially equivalent to estimate the operator
norm of $P_{\lambda,\delta}$, or that of the smoothed out version
$$
P_{\lambda,\delta}^\flat = \chi \left( \frac{\sqrt{-\Delta} - \lambda}{\delta} \right),
$$
where $\chi$ is a cutoff function; we will use indiscriminately all these formulations\footnote{Up to logarithmic factors, it is also equivalent to consider boundedness of the resolvent operator $(\Delta - z)^{-1}$, $\mathfrak{Im} z \neq 0$.}.

\bigskip

A beautiful textbook presentation of the theory for $\delta=1$ can be found in~\cite{Sogge1}, see also~\cite{SoggeHangzhou,Zelditch0,Zelditch}. In this review note, we will try to give an overview of progress on the case $\delta<1$, in particular in the case where $\delta$ is polynomially small in $\lambda$.

A very closely connected topic is that of Quantum Unique Ergodicity for negatively curved manifolds, or for general manifolds the description of possible weak-$L^2$ limits of eigenfunctions. This is in many ways an $L^2$ version of the $L^p$-problems which will occupy us in the present review, with much of the same underlying phenomenology. We refer to~\cite{Anantharaman,Sarnak} for surveys on Quantum Unique Ergodicity.

\section{Sogge's universal estimates for $\delta = 1$}

\subsection{Sogge's theorem} In this subsection, we consider general compact manifolds $M$, without any further restriction. Denoting by $(\lambda_j)$ and $(\varphi_j)$ the eigenvalues and eigenvectors of the Laplacian on $M$, H\"ormander's bound~\cite{Hormander}, see also~\cite{Avakumovic,Levitan} states that
\begin{equation}
\label{classicalbound}
\| \varphi_j \|_{L^\infty} \lesssim \lambda_j^{\frac{d-1}{2}} \| \varphi_j \|_{L^2}
\end{equation}
(this estimate arose in connection with the Weyl law, to which we will come back in Section~\ref{sectionweyl}).

This classical bound was subsumed in the following much stronger result of Sogge: it covers any $p>2$, allows for spectral projectors instead of eigenfunctions, and it is not only providing an upper bound, but also a lower bound. It will be convenient to denote
\begin{equation}
\label{formulagamma}
\gamma(p) = \max \left[ \frac{d-1}{2} - \frac{d}{p}\,,\,\frac{d-1}{2} \left( \frac{1}{2} - \frac{1}{p} \right) \right],
\end{equation}
since this exponent will appear throughout the present review.

\begin{thm}[Sogge] If $M$ is a complete Riemannian manifold with bounded geometry, then
\begin{equation}
\label{universalsogge}
\| P_{\lambda,1} \|_{L^2 \to L^p} \lesssim \lambda^{\gamma(p)} \qquad \mbox{if $2 \leq p \leq \infty$}.
\end{equation}
Furthermore, there exists a constant $R_0$ such that: for any $\lambda_0$, there exists $\lambda$ with $|\lambda - \lambda_0| < R_0$ and
$$
\| P_{\lambda,1} \|_{L^2 \to L^p} \gtrsim \lambda^{\gamma(p)} \qquad \mbox{if $2 \leq p \leq \infty$}.
$$
\end{thm}

The exponent $\gamma(p)$ is defined as a maximum in~\eqref{formulagamma}, and which of the two terms in this maximum is dominant depends on how $p$ compares compares to the Stein-Tomas exponent defined by
$$
p_{ST} = \frac{2(d+1)}{d-1}.
$$
More precisely
\begin{itemize}
\item For $p \geq p_{ST}$, $P_{\lambda,1}$ is in the \textit{point-focusing regime}: the term $\lambda^{\frac{d-1}{2} - \frac{d}{p}}$ is dominant, and examples of functions which focus at a point can be shown to achieve this value (up to a bounded factor) for $\frac{\| P_{\lambda,1} f \|_{L^p}}{\| f \|_{L^2}}$. 
\item For $2 \leq p \leq p_{ST}$, $P_{\lambda,1}$ is in the \textit{geodesic-focusing regime}: the term
$\lambda^{\frac{d-1}{2} \left( \frac{1}{2} - \frac{1}{p} \right)}$ is dominant. Functions can be constructed which focus on a geodesic and achieve this value (up to a bounded factor) for $\frac{\| P_{\lambda,1} f \|_{L^p}}{\| f \|_{L^2}}$.
\end{itemize}

\begin{figure}
  \includegraphics[width=12 cm]{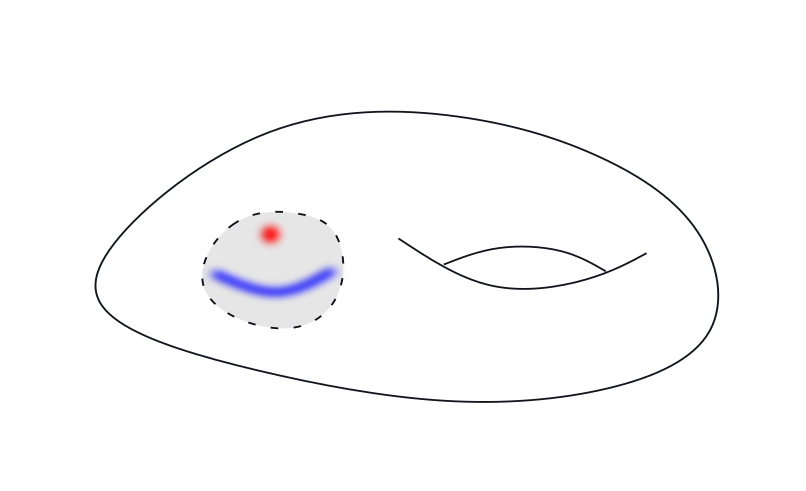}
\caption{Consider a coordinate patch on a Riemannian manifold. One can then contruct spectrally localized functions ($P_{\lambda,\delta} f = f$) which concentrate their $L^p$ mass in a neighborhood of a given point (red), or a given geodesic (blue). These functions saturate the value $\lambda^{\gamma(p)}$ for the operator norm in~\eqref{universalsogge} up to a bounded factor in the regimes $p>p_{ST}$ and $p<p_{ST}$, respectively. }
\end{figure}

\begin{proof}[Proof of Theorem~\ref{universalsogge}] A detailed proof can be found in~\cite{Sogge0,Sogge1} in the compact case, see also~\cite{AGL} for the simple extension to the complete case. A sketch of the upper bound is as follows.

First, one can reduce matters to an operator $\widetilde{P}_{\lambda,1}$ which is a slightly modified version of $P_{\lambda,1}$, simpler to estimate since it is local. Through pseudodifferential calculus, one can then obtain a very precise description of the kernel of $\widetilde{P}_{\lambda,1}$:
\begin{equation}
\label{kernelK}
\widetilde{P}_{\lambda,1} f (x) = \int_M K_{\lambda,1}(x,y) f(y) \,dx, \qquad K_{\lambda,1}(x,y) = \lambda^{\frac{d-1}{2}} a(x,y) e^{i \lambda \psi (x,y)} + O(\lambda^{-\infty}),
\end{equation}
where $a \in \mathcal{C}_0^\infty$. The desired result follows by interpolation between the following Lebesgue exponents
\begin{itemize}
\item For $p=2$, the bound $O(1)$ is an immediate consequence of the definition of $\widetilde{P}_{\lambda,1}$.
\item For $p=\infty$, the bound $O(\lambda^\frac{d-1}{2})$ follows immediately from the formula ~\eqref{kernelK} and the Cauchy-Schwarz inequality.
\item Finally, $p = p_{ST}$ is the critical exponent. The properties of the phase $\psi(x,y)$ now come into play: indeed, it satisfies the Carleson-Sj\"olin condition, which, roughly speaking, requires that its Hessian has rank $n-1$. The $L^2 \to L^p$ bounds on Carleson-Sj\"olin operators (due in dimension 2 to Carleson-Sj\"olin~\cite{CS} and H\"ormander~\cite{Hormander1}, and in higher dimension to Stein~\cite{SteinBeijing}) then give the desired result. These bounds can be thought of as a variable coefficient generalization of the Stein-Tomas theorem, which will presented in Section~\ref{Euclidean}.
\end{itemize}
\end{proof}

As observed in~\cite{AnkerGermainLeger}, a consequence of the optimality of the Sogge estimates is the universal lower bound
\begin{equation}
\label{universallowerbound}
\| P_{\lambda,\delta} \|_{L^2 \to L^p} \gtrsim  \lambda^{\gamma(p)} \delta^{\frac{1}{2}} \qquad \mbox{if $2 \leq p \leq \infty$}
\end{equation}
(more precisely: for any $\lambda_0,\delta$, there exists $\lambda$ such that $|\lambda - \lambda_0| < R_0$, and such that the above holds).

\subsection{The case of the sphere}

As we saw, the universal Sogge estimate~\eqref{universalsogge} corresponds to the case $\delta =1$, and is sharp for any manifold. Is it possible to improve this bound by taking $\delta$ smaller? 
The example of the sphere shows that the answer is negative in general. Indeed, the spectrum of the Laplace-Beltrami operator on $\mathbb{S}^d$ is given by $\{ k(k+d-1),\,k \in \mathbb{Z}\}$, so that
$$
P_{\lambda,\delta} = P_{\lambda,\frac{1}{2}} \qquad \mbox{if} \;\; \delta < \frac{1}{2}, \;\; \lambda = \sqrt{k(k+d-1)}, \; k \in \mathbb{Z}
$$
(and $\lambda$ sufficiently big). This implies that operator bounds for $P_{\lambda,\delta}$ and $P_{\lambda, \frac{1}{2}}$ are equal and that the universal Sogge estimate~\eqref{universalsogge} will be saturated by eigenfunctions; the two following examples appear in~\cite{Sogge00} and are fundamental.
\begin{itemize}
\item In the point-focusing regime, \textit{zonal spherical harmonics} saturate~\eqref{universalsogge}. By definition, zonal functions can be written as functions of the geodesic distance to a point on the sphere. Each eigenspace contains a zonal function, unique up to the symmetry group $SO(d+1)$. Normalizing this function in $L^2$,  it reaches a size $\sim \lambda^{\frac{d-1}{2}}$ on a set of diameter $\sim \lambda^{-1}$ around the pole, leading to the estimate $\| f \|_{L^p} \sim \lambda^{\frac{d-1}{2} - \frac{d}{p}} \| f \|_{L^2} $ for $p > p_{ST}$.
\item In the geodesic-focusing regime, \textit{highest weight spherical harmonics} saturate~\eqref{universalsogge}. Viewing the sphere $\mathbb{S}^{d-1}$ as embedded in $\mathbb{R}^{d+1}$, they can be taken to be the trace on the sphere of $(x_1 + i x_2)^k$. Normalizing this function in $L^2$, it reaches a size $\sim \lambda^{\frac{d-1}{4}}$ on a neighbourhood of size $\sim \lambda^{-\frac 12}$ around the great circle $x_1^2 + x_2^2 =1$, leading to the estimate $\| f \|_{L^p} \sim \lambda^{\frac{d-1}{2}\left(\frac{1}{2} - \frac{1}{p}\right)}$ for $p < p_{ST}$.
\end{itemize} 

\begin{figure}
  \includegraphics[width=12 cm]{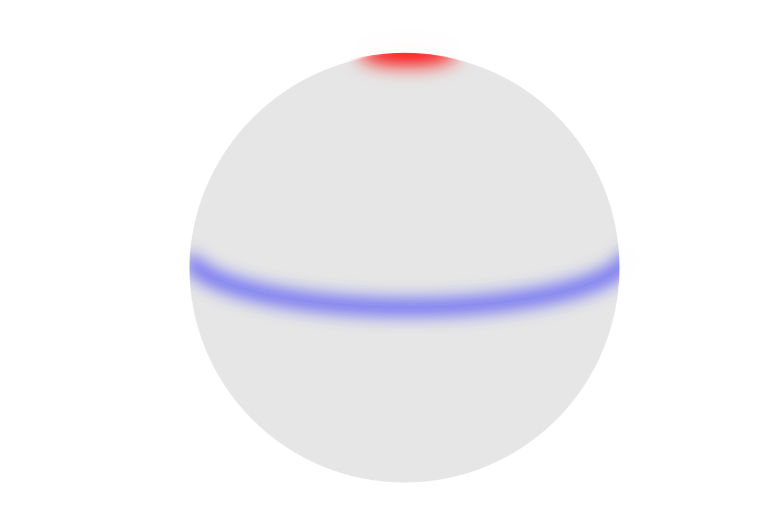}
\caption{On the sphere, zonal spherical harmonics concentrate most of their $L^p$ mass close to the North pole (in red), and highest weight spherical harmonics concentrate most of their $L^p$ mass close to the equator (in blue). These eigenfunctions of the Laplacian saturate the Sogge bound~\eqref{universalsogge} for $p>p_{ST}$ and $p<p_{ST}$ respectively. Of course, the North pole and the equator do not play a distinguished role, since the Laplacian is invariant under $O(d+1)$, and this figure can be rotated.}
\end{figure}

\subsection{When are Sogge's estimates optimal for eigenfunctions?} We saw that, in the case of the sphere, the Sogge bound~\eqref{universalsogge} agrees with the optimal bound for eigenfunctions, or in other words that the spectral projectors for $\delta=0$ and $\delta=1$ have comparable operator norms. Is it possible to characterize all manifolds for which this phenomenon occurs?

\label{soggeoptimal?}

\subsubsection{Point-focusing regime} A line of research initiated by Sogge and Zelditch ~\cite{SZ,SZ1,SZ2, SoggeXEDP} and pursued by Canzani and Galkowski~\cite{CG1,CG0} seeks conditions on a compact manifold $M$ under which there exists, or not, a sequence of eigenfunctions $\varphi_j$ associated to eigenvalues $\lambda_j$ exhibiting \textit{maximum eigenfunction growth}, namely
\begin{equation}
\label{focusingeigenfunction}
\| \varphi_j \|_{L^\infty} \sim \lambda_j^{\frac{d-1}{2}} \| \varphi_j \|_{L^2}.
\end{equation}
Of course, this implies that $\| P_{\lambda,\delta} \|_{L^2 \to L^\infty} \gtrsim \lambda^{\frac{d-1}{2}}$ for any $\delta>0$, along a sequence of $\lambda$.

While very powerful and precise results have been proved along these lines, we will give here a simple example of such a statement, which can be found in~\cite{STZ}. For $x \in M$, define $\mathcal{L}_x \subset S_x^*M$ (unit cotangent bundle) to be the set of directions $\xi \in T_x^*M$ such that the geodesic $\gamma$ with initial data $\gamma(0) = x$, $\dot \gamma(0) = \xi$ returns to $x$. Then there cannot exists a sequence of eigenfunctions with maximal eigenfunction growth unless $\mathcal{L}_x$ has positive measure for some $x \in M$.

This is a strong constraint, of course reminiscent of the case of the sphere; but it is not stable under perturbation. Indeed, it was showed~\cite{SZ} that for generic metrics, $\| \varphi_j \|_{L^\infty} = o( \lambda_j^{\frac{d-1}{2}})$.

Finally, we mention the original point of view in~\cite{Steinerberger}, which relates poinwise $L^\infty$ growth to correlations between eigenfunctions at different points of the manifold.

\subsubsection{Geodesic-focusing regime} We call a geodesic stable if its associated Poincar\'e map only has eigenvalues of modulus 1. If there exists a stable closed geodesic, one can construct a sequence $(\lambda_j)$ going to infinity and associated \textit{quasimodes} $(f_j)$ such that
$$
\left\| \left[ \Delta + \lambda_j^2 \right] f_j \right\|_{L^2} \lesssim_N \lambda_j^{-N} \qquad \mbox{for any $j,N \in \mathbb{N}$;}
$$
furthermore, these quasimodes concentrate within $\frac{1}{\sqrt{\lambda_j}}$ of the geodesic, and are bounded pointwise by $\lambda_j^{\frac{d-1}{4}}$. To the best of our knowledge, this exact statement does not appear in the literature, but should follow from the construction of quasimodes supported on stable geodesics~\cite{GW,Ralston,Donnelly,ACMT}.

As a consequence of the existence of these quasimodes, the spectral projectors can be bounded from below by
$$
\| P_{\lambda_j,\delta} \|_{L^2 \to L^p} \gtrsim_N \lambda_j^{\frac{d-1}{2} \left( \frac{1}{2} - \frac{1}{p} \right)} \qquad \mbox{if $\delta > \lambda_j^{-N}$}.
$$
Note that the right-hand side matches the universal bound of Sogge in the regime $p < p_{ST}$.

\bigskip

The preceding discussion suggests that, for "most" manifolds, the bounds on $\| P_{\lambda,\delta} \|_{L^2 \to L^p}$ should improve as $\delta$ decreases from $\delta=1$. In the remainder of this review article, we will try and characterize this improvement in some specific examples.

\section{The Euclidean space and its quotients}

\label{Euclidean}

\subsection{The Euclidean space} Consider the extension operator mapping functions on the sphere $\mathbb{S}^{d-1}$ to functions on the whole space $\mathbb{R}^d$ through the formula
$$
f \mapsto \mathcal{F} \left[ f \,d\sigma_{|\mathbb{S}^{d-1}} \right],
$$
where $\mathcal{F}$ is the Fourier transform and $d\sigma_{|\mathbb{S}^{d-1}}$ is the surface measure on $\mathbb{S}^{d-1}$. The Stein-Tomas theorem establishes the boundedness of this operator from $L^2(\mathbb{S}^{d-1})$ to $L^p(\mathbb{R}^d)$ if $p \geq p_{ST}$. Recently, several works have investigated the existence and structure of functions which saturate the operator norm~\cite{Shao,FO}.
The Stein-Tomas theorem can be reformulated in terms of our spectral projectors:

\begin{thm}[Stein-Tomas~\cite{Stein,Tomas}]
For any $\lambda>1$, $\delta<1$,
\begin{equation}
\label{SteinTomasinequality}
\| P_{\lambda,\delta} \|_{L^2 \to L^p} \sim \lambda^{\frac{d-1}{2} - \frac{d}{p}} \delta^{\frac{1}{2}} + \lambda^{\frac{d-1}{2} \left( \frac{1}{2} - \frac{1}{p} \right)}\delta^{\frac{d+1}{2} \left( \frac{1}{2} - \frac{1}{p} \right)} \qquad \mbox{if $2 \leq p \leq \infty$}.
\end{equation}
\end{thm}

Before sketching the proof of this theorem, we remark that the Stein-Tomas estimate above can be extended to asymptotically conic, non-trapping manifolds~\cite{GHS}; extensions to non-elliptic, constant coefficients operators can also be considered~\cite{KRS}.

\begin{proof} By scaling, this theorem is a consequence of the universal Sogge estimate~\eqref{universalsogge}; but it came first, and can be proved by simpler means. We sketch the proof in dimension 3, where formulas are cleaner. The smoothed out operator $P^\chi_{\lambda,\delta}$ defined as $\chi\left( \frac{\sqrt{-\Delta} - \lambda}{\delta} \right)$ (where $\chi \in \mathcal{C}_0^\infty$) can be expressed through a radial convolution kernel
\begin{equation}
\label{euclideankernel}
P^\chi_{\lambda,\delta} f (x) = \int_{\mathbb{R}^3} K_{\lambda,\delta}(|x-y|) f(y) \, dy, \qquad \mbox{with} \qquad
K_{\lambda,\delta}(r) = \delta \widehat{\chi} (\delta r) \lambda \frac{\sin(\lambda r)}{r}.
\end{equation}
The proof consists in splitting $K_{\lambda,\delta}$ into
$$
K_{\lambda,\delta}(r) = \sum_j K_{\lambda,\delta}^j (r), \qquad \mbox{and accordingly} \qquad P_{\lambda,\delta} = \sum_j P_{\lambda,\delta}^j,
$$
where each $K_{\lambda,\delta}^j$ is smoothly localized where $r \sim 2^j$. Then the operator norm of $P_{\lambda,\delta}^j$ can be estimated in two ways
\begin{itemize}
\item $\| P_{\lambda,\delta}^j \|_{L^1 \to L^\infty} \lesssim \| K_{\lambda,\delta}^j \|_\infty$
\item $\| P_{\lambda,\delta}^j \|_{L^2 \to L^2} \lesssim \| m_{\lambda,\delta}^j \|_\infty$, where $m_{\lambda,\delta}^j$ is the Fourier symbol of $P_{\lambda,\delta}^j$.
\end{itemize}
By interpolation, this gives a bound for $\| P_{\lambda,\delta}^j \|_{L^{p'} \to L^p}$, and, after filling out the estimates and summing over $j$, the desired bound for $\| P_{\lambda,\delta} \|_{L^{p'} \to L^p}$ if $p \neq p_{ST}$. The critical point $p_{ST}$ requires complex interpolation.
\end{proof}

Just like for Sogge's estimate~\eqref{universalsogge}, one can distinguish two regimes in~\eqref{SteinTomasinequality}:
\begin{itemize}
\item In the point-focusing regime $p \geq p_{ST}$, the dominant term is $\lambda^{\frac{d-1}{2} - \frac{d}{p}} \delta^{\frac{1}{2}}$. It is achieved up to a bounded factor for the \textit{spherical example} $f$ (the terminology coming from the radial symmetry of $f$) given in Fourier by $\widehat{f}(\xi) = \mathbf{1}_{[\lambda-\delta,\lambda+\delta]}(|x|)$, which has a sharp peak at the origin.
\item In the geodesic-focusing regime $p \leq p_{ST}$, the dominant term is $\lambda^{\frac{d-1}{2} \left( \frac{1}{2} - \frac{1}{p} \right)}\delta^{\frac{d+1}{2} \left( \frac{1}{2} - \frac{1}{p} \right)}$; it can be achieved up to a bounded factor by the so-called \textit{Knapp example} which we now describe. It is the function defined, in Fourier space, by $\widehat{f} = \mathbf{1}_R$, where $R$ is a rectangle of dimensions $\sim (\sqrt{\lambda \delta})^{d-1} \times \delta$ contained in the annulus $\{ \xi, \, \lambda - \delta < |\xi| < \lambda + \delta \}$. In physical space, it is mostly supported on the dual rectangle $\widehat{R}$, of dimensions $((\lambda \delta)^{-\frac 12})^{d-1} \times \delta$.
\end{itemize}
These two examples are depicted in Figure~\ref{euclideanpicture}, both in physical space and in Fourier space.

\begin{figure}
  \includegraphics[width=12 cm]{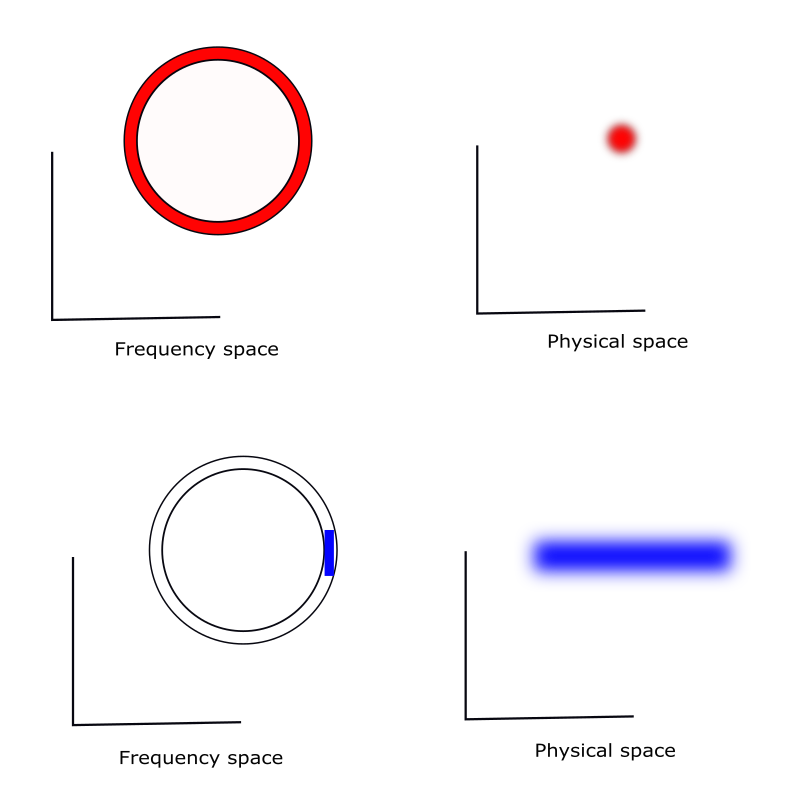}
\caption{\label{euclideanpicture}
The spherical example (red) and the Knapp example (blue) for the Euclidean space represented in frequency and physical space: the Fourier transform exchanges the left and right columns. In frequency space, the supports are contained within the annulus with inner and outer radii $\lambda-\delta$ and $\lambda + \delta$, either the full annulus or a rectangle contained in it. In physical space, the $L^p$ mass concentrates at a point or on the dual rectangle, respectively.}
\end{figure}

\subsection{Tori} We now turn to quotients of the Euclidean space, namely tori. As in the whole space, we can rely on Fourier analysis (Fourier series rather than the Fourier transform) to carry out the analysis, but as we shall see, it becomes much more challenging.

Choosing a family of independent vectors $(e_i)$ of $\mathbb{R}^d$ generating  the lattice $\Lambda = \mathbb{Z} e_1 + \dots + \mathbb{Z} e_d$, we will denote $\mathbb{T}^d = \mathbb{T}^d_\Lambda$ for the quotient $\mathbb{R}^d / (\mathbb{Z} e_1 + \dots + \mathbb{Z} e_d)$.

\subsubsection{The case $p= \infty$} A basis of eigenfunctions of the Laplacian on $\mathbb{T}^d$ is provided by complex exponentials of the type $e^{ik \cdot x}$, where $k$ belongs to the dual lattice $\Lambda^*$. It is then easy to see that
$$
\| P_{\lambda,\delta} \|_{L^1 \to L^\infty} = \# \{ k \in \Lambda^*, \, \lambda - \delta < |k| < \lambda + \delta \},
$$
which brings us to a lattice point counting problem. This is a variant of the more classical problem of estimating the number of lattive points in a sphere of radius $\lambda$:
$$
N(\lambda) = \# \{ k \in \Lambda^*, \, |k| < \lambda \}.
$$
To leading order, $N(\lambda)$ equals $C_\Lambda \lambda^{d}$, and the problem is to bound the error $E(\lambda) = \left| P(\lambda) - C_\Lambda \lambda^{d} \right|$. A full understanding of optimal bounds on $E(\lambda)$ is still missing in dimensions 2 and 3; we refer the reader  to~\cite{Goetze,Huxley,Kraetzel,Nowak} for entry points into the literature.

Though bounds on $E(\lambda)$ imply bounds on $\| P_{\lambda,\delta} \|_{L^1 \to L^\infty}$ by the above characterization, this cannot lead to an optimal answer, and the questions of estimating $E(\lambda)$ and $\| P_{\lambda,\delta} \|_{L^1 \to L^\infty}$ are not equivalent.

\subsubsection{The case $\delta = 0$ (eigenfunctions)} The earliest $L^p$ bound on eigenfunctions of the Laplacian on the torus is the famous result of Cooke-Zygmund~\cite{Cooke,Zygmund}  
that, on $\mathbb{T}^2$, an eigenfunction $\varphi$ of the Laplacian satisfies
$$
\| \varphi \|_{L^4} \lesssim \| \varphi \|_{L^2}.
$$
Whether this estimate can be extended to any $p>4$ has remained an open problem. If $\Lambda = \mathbb{Z}^d$, Bourgain~\cite{Bourgain1} conjectured the general bound for an eigenfunction $\varphi$ associated to the eigenvalue $\lambda$
\begin{equation}
\label{conjecturebourgain}
\| \varphi \|_{L^p} \lesssim \lambda^{\frac{d}{2} - 1 - \frac{d}{p}} \| \varphi \|_{L^2} \qquad \mbox{for $p > p^* = \frac{2d}{d-2}$}
\end{equation}
This conjecture was the subject of a number of subsequent papers~\cite{Bourgain2,BD1,BD2} culminating in the $\ell^2$ decoupling estimate of Bourgain and Demeter~\cite{BD3}. Though the $\ell^2$ decoupling theorem was a landmark result in harmonic analysis, it only proved the conjecture in the range $p \geq \frac{2(d-1)}{d-3}$.

\subsubsection{General values of $p$ and $\delta$}
The conjecture~\eqref{conjecturebourgain} gives $L^p$ bounds on eigenfunctions of the Laplacian on the torus. This is meaningful if $\Lambda = \mathbb{Z}^d$, but this question becomes trivial on generic tori, for which eigenspaces of the Laplacian have uniformly bounded dimension. A natural generalization consists in asking the question which is at the heart of the present review article, namely that of estimating $\| P_{\lambda,\delta} \|_{L^2 \to L^p}$. The author and Myerson~\cite{GM1} proposed the following conjecture\footnote{Notice that it implies the conjecture~\eqref{conjecturebourgain} by choosing $\Lambda = \mathbb{Z}^d$ and $\delta \sim \lambda^{-1}$.}
\begin{equation}
\label{conjectureGM}
\| P_{\lambda,\delta} \|_{L^2 \to L^p} \lesssim \lambda^{\frac{d-1}{2} - \frac{d}{p}} \delta^{\frac{1}{2}} + (\lambda \delta)^{\frac{d-1}{2} \left( \frac{1}{2} - \frac{1}{p} \right)} \qquad \mbox{if $2 \leq p \leq \infty$ and $\delta > \lambda^{-1}$}.
\end{equation}

As in the case $\delta = 1$, and in the case of the Euclidean space, the two summands on the right-hand side are achieved up to a bounded factor for eigenfunctions focusing at a point, or on a geodesic. However, the dominant term does not depend on how $p$ compares to $p_{ST}$; it is rather a condition depending on $p$, $\lambda$, and $\delta$.

\begin{itemize}
\item In the point-focusing regime, the dominant term is $\lambda^{\frac{d-1}{2} - \frac{d}{p}} \delta^{\frac{1}{2}}$. This operator norm for $P_{\lambda,\delta}$ can be achieved up to a bounded factor for functions which are evenly spread amongst all Fourier modes $k \in \Lambda^*$ such that $\lambda - \delta < |k| < \lambda + \delta$: this is the torus analog of the spherical example.
\item In the geodesic-focusing regime, the dominant term is $(\lambda \delta)^{\frac{d-1}{2} \left( \frac{1}{2} - \frac{1}{p} \right)}$. This can be achieved up to a bounded factor for functions which, in physical space, are focusing on a closed geodesic of $\mathbb{T}^d$: this is the torus analog of the Knapp example.
\end{itemize}

The spherical and Knapp examples for the torus are depicted in Figure~\ref{torusfigure}, both in physical and in frequency space.

\begin{figure}
  \includegraphics[width=12 cm]{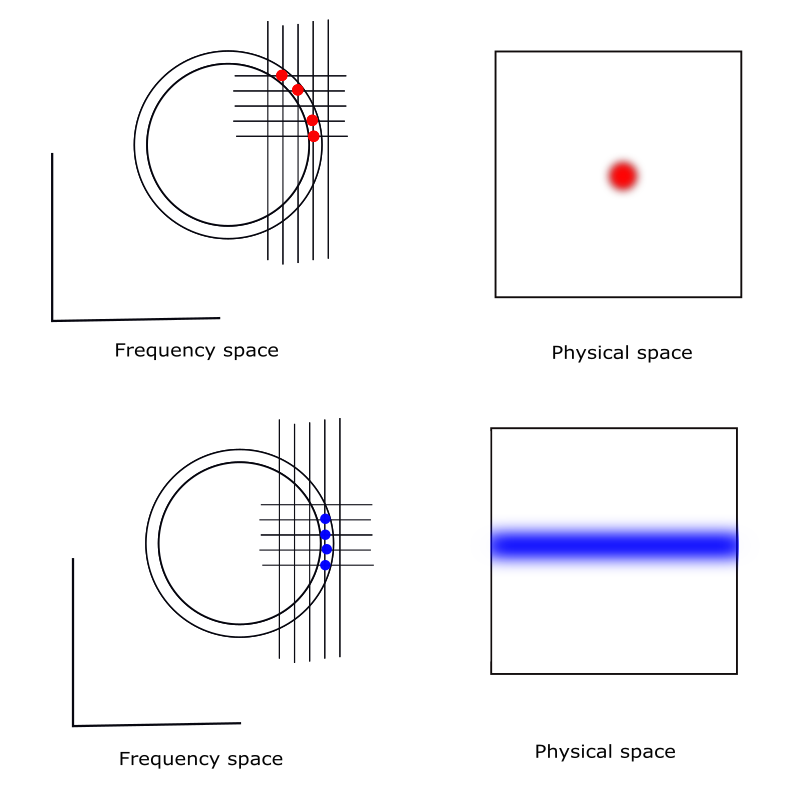}
\caption{\label{torusfigure} The spherical example (red) and the Knapp example (blue) for the torus represented in frequency and physical space: the Fourier series exchange the left and right columns. In frequency space, the supports are contained within the annulus with inner and outer radii $\lambda-\delta$ and $\lambda + \delta$, either the full annulus or a rectangle contained in it. In physical space, the $L^p$ mass concentrates close at a point or on the dual rectangle, respectively.}
\end{figure}

\bigskip

In the following theorem, we do not attempt to gather all the cases for which the conjecture is known to be true, but rather record two relevant instances where it holds. For the current state of the art, the interested reader can consult~\cite{DemeterGermain,GM3,Hickman}, besides the papers which have already been mentioned.

\begin{thm}
The conjecture~\eqref{conjectureGM} is verified
\begin{itemize}
\item[(i)] for $\Lambda = \mathbb{Z}^d$, if $p \geq \frac{2(d-1)}{d-3}$ and $d \geq 4$
\item[(ii)] for general $\Lambda$, if $p \leq p_{ST}$.\end{itemize}
\end{thm}

\begin{proof} The first assertion is due to Bourgain and Demeter~\cite{BD3}; its proof combines the $\ell^2$ decoupling theorem with arithmetic properties of the lattice $\mathbb{Z}^d$.

The second assertion turns out to be a rather direct consequence of $\ell^2$ decoupling. To simplify notations, we focus on the case $d=2$, for which $p_{ST}=6$, and choose the lattice to be $\Lambda = \mathbb{Z}^2$. By interpolation with the trivial case $p=2$, it suffices to bound the $L^6$ norm of a function $f \in L^2(\mathbb{T}^2)$ which is Fourier supported on an annulus of inner radius $\lambda-\delta$ and outer radius $\lambda + \delta$. We write $f$ as the sum of its Fourier series $f = \sum a_k e^{2\pi i k \cdot x}$ and rescale variables to $X = \lambda x$ and $K = k/ \lambda$ in order to apply the $\ell^2$ decoupling theorem:
\begin{align*}
\| f \|_{L^6(\mathbb{T}^2)} & = \left\| \sum_{k \in \mathbb{Z}^2} a_k e^{2\pi i k \cdot x} \right\|_{L^6(\mathbb{T}^2)} 
 \lesssim \left( \frac{\delta}{\lambda} \right)^{\frac 13} \left\| \phi \left( \frac{\delta X}{\lambda} \right) \sum_{K \in \mathbb{Z}^2/\lambda}  a_{\lambda K} e^{2\pi i K \cdot X} \right\|_{L^{6}(\mathbb{R}^2)},
\end{align*}
where the cutoff function $\phi$ can be chosen to have compactly supported Fourier transform. As a result, the Fourier transform of the function on the right-hand side is supported on a $\delta/\lambda$-neighborhood of $\mathbb{S}^{d-1}$. It can be written as a sum of functions which are supported in Fourier on caps $\theta$ with dimensions $\sim \frac{\delta}{\lambda} \times \frac{\delta^{\frac 12}}{\lambda^{\frac 12}}$. These caps are almost disjoint rectangles which cover the annulus, and to which we associate a smooth partition of unity $(\chi_\theta)$. By $\ell^2$ decoupling, the $L^6$ norm above is bounded by
$$
\lesssim_\epsilon \left( \frac{\delta}{\lambda} \right)^{\frac 13 - \epsilon} \left( \sum_{\theta} \left\|  \phi \left( \frac{\delta X}{\lambda} \right)  \sum_{K \in \mathbb{Z}^2 / \lambda} \chi_\theta( K) a_{\lambda K} e^{2\pi i K \cdot X}\right\|_{L^6(\mathbb{R}^2)}^2  \right)^{\frac 12}.
$$
At this point, we use the inequality $ \| g \|_{L^p(\mathbb{R}^2)} \lesssim \| g \|_{L^2(\mathbb{R}^2)} | \operatorname{Supp} \widehat{g} |^{\frac{1}{2} - \frac{1}{p}}$, which is valid for $p \geq 2$ and follows by applying successively the Hausdorff-Young, H\"older, and the Plancherel inequalities. We use this inequality for $g_\theta(X) =  \phi \left( \frac{\delta X}{\lambda} \right)  \sum_{K}  \chi_\theta(K) a_{\lambda K} e^{2\pi i K \cdot X}$. Since each cap contains at most $\lambda^{\frac 12} \delta^{\frac 12}$ rescaled lattice points, its Fourier transform is supported on the union of at most $\lambda^{\frac 12} \delta^{\frac 12}$ balls of radius $O( \delta / \lambda) $, giving $| \operatorname{Supp} \widehat{f} | \lesssim \lambda^{-\frac 12}\delta^{\frac 32} $. Thus the $L^6$ norm we are trying to bound is less than
\begin{align*}
 & \lesssim  \left( \frac{\delta}{\lambda} \right)^{\frac 1 3 -\epsilon} (\lambda^{-\frac 12}\delta^{\frac 32})^{\frac 13} 
 \left( \sum_{\theta} \left\|  \phi \left( \frac{\delta X}{\lambda} \right)  \sum_{K \in \mathbb{Z}^2 / \lambda} \chi_\theta(\lambda K) a_{\lambda K} e^{2\pi i K \cdot X}\right\|_{L^2(\mathbb{R}^2)}^2  \right)^{\frac 12} 
 \end{align*}
 By almost orthogonality and periodicity of the Fourier series, this is bounded by $\lambda^{\epsilon} (\lambda \delta)^{\frac 16} \| f \|_{L^2(\mathbb{T}^2)}$, which is the desired estimate.
\end{proof}

Before completing this section on Euclidean tori, we mention two related directions. First, the two-dimensional Euclidean cylinder is considered in~\cite{GM2}; in this case, there are no difficulties related to number theoretical questions, and the conjecture~\eqref{conjectureGM} could be proved. Second, Euclidean tori are the simplest examples of manifolds with integrable geodesic flow; this broader class was investigated in~\cite{Bourgain0,TothZelditch}.

\section{The hyperbolic space and its quotients}

\subsection{Eigenfunctions of compact manifolds of nonpositive curvature}

Manifolds of nonpositive curvature form a much broader class than the hyperbolic space and its quotients (it even includes Euclidean tori), and a natural context in which improvements to the H\"ormander bound~\eqref{classicalbound} are expected. A long line of research~\cite{Berard,HR,HT,BS1,BS2,BS} focused on this question through microlocal analysis, ultimately resulting in the following theorem.

\begin{thm} If the manifold $M$ is compact and has nonnegative curvature, then, denoting its eigenfunctions and eigenvalues by $(\varphi_j,\lambda_j)$, there holds
$$
\| \varphi_j \|_{L^p} \lesssim \lambda_j^{\gamma(p)} (\log \lambda_j)^{-\kappa},
$$
where $\kappa = \kappa(p) =\frac{1}{2}$ for $p>p_{ST}$, and $\kappa > 0$ for $p \in (2,p_{ST}]$.
\end{thm} 

The above theorem is fundamental in the analysis of the spectral resolution of the Laplacian: it deals with the hardest case, that of compact manifolds, and it does so without any symmetry assumptions, under the single geometric condition of nonpositive curvature. Similar improvements can actually be obtained in even more general contexts, see~\cite{CG0,CG2,CG3}.

However, stronger bounds are expected: under chaoticity assumptions for the geodesic flow, the \textit{random wave model} of Berry~\cite{Berry} leads to predicting a bound $\sqrt{\log \lambda_j}$ (see for instance~\cite{Steinerberger} for the heuristic argument, and~\cite{ABST} for numerics which appear to support this conjecture).

Such a bound seems utterly out of reach of current analytic methods. However, in the hope of achieving some progress, we will now shift the focus and consider a more restricted class of examples, namely the hyperbolic space and its quotients, allowing also for infinite volume.

\subsection{The hyperbolic space} The question of estimating $\| P_{\lambda,\delta} \|_{L^2 \to L^p}$ can be fully answered for the hyperbolic space. After partial results were obtained in~\cite{HuangSogge0}, this was done in~\cite{CH}, and the optimal dependence of the bounds on $p$ was proved in~\cite{GL}.

\begin{thm}\label{thmhyperbolic} If $M$ is the hyperbolic space $\mathbb{H}^d$,
$$
\| P_{\lambda,\delta} \|_{L^2 \to L^p} \lesssim [(p-2)^{-1} + 1 ] \lambda^{\gamma(p)} \delta^{\frac{1}{2}}.
$$
\end{thm}

The reader should by now not be surprised that the upper bound in this theorem comprises two regimes, as reflected in the definition of $\gamma(p)$ in~\ref{formulagamma}: these are of course the point-focusing and the geodesic-focusing regime. Just like in the Euclidean case, a spherical example and a Knapp example can be constructed, which nearly saturate the bounds in the point-focusing and geodesic-focusing regimes respectively. These examples are depicted in Figure~\ref{figurehyperbolic}, both in physical space (hyperbolic space) and in frequency space (the appropriate analog of the Fourier transform being provided by the Helgason-Fourier transform).

\begin{figure}
  \includegraphics[width=12 cm]{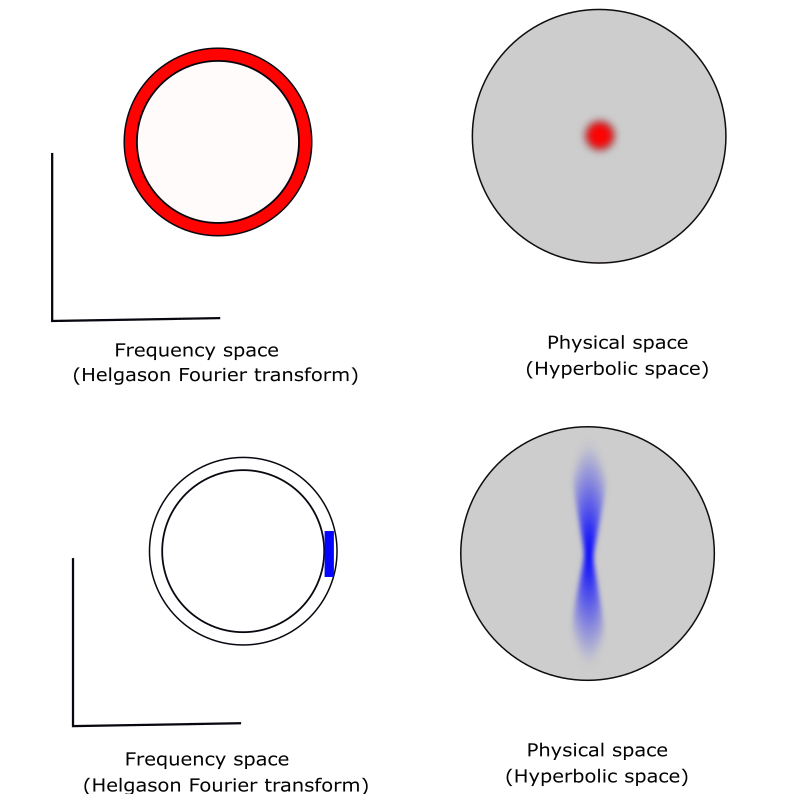}
\caption{\label{figurehyperbolic}
The spherical example (red) and the Knapp example (blue) for the hyperbolic space represented in frequency and physical space. On the physical side, we use the Poincar\'e disk model to represent the hyperbolic space; and on the frequency side, it is the Fourier-Helgason transform which provides the dual picture. In frequency space, the supports are contained within the annulus with inner and outer radii $\lambda-\delta$ and $\lambda + \delta$, either the full annulus or a rectangle contained in it. In physical space, the $L^p$ mass concentrates at a point, or along a geodesic; notice how the instability of geodesics on the hyperbolic space leads to spreading at the ends of the Knapp example.}
\end{figure}

 Comparing this bound to the Euclidean case or the sphere, it appears that the strongest bound is reached in the hyperbolic case, the weakest in the spherical case, with the Euclidean case being a middle ground. This can be explained, at least partially, in relation to the stability of geodesics: the more stable geodesics are, the easier it is for almost-eigenfunctions to focus on them, and therefore the weaker the bounds for $\| P_{\lambda,\delta} \|_{L^2 \to L^p}$.

\begin{proof}[Proof of Theorem~\ref{thmhyperbolic}] We sketch the proof in dimension 3, where formulas are cleaner. The operator $P_{\lambda,\delta}$ can then be expressed through a radial convolution kernel
\begin{equation*}
P_{\lambda,\delta} f (x) = \int_{\mathbb{H}^3} K_{\lambda,\delta}(\operatorname{dist}(x,y)) f(y) \, dy, \qquad \mbox{with} \qquad
K_{\lambda,\delta}(r) = \delta \widehat{\chi} (\delta r) \lambda \frac{\sin(\lambda r)}{\sinh r}.
\end{equation*}
This is quite similar to the Euclidean case~\eqref{euclideankernel} over small distances, but over large distances, the kernel decays exponentially. Nevertheless, it is possible to follow closely the scheme of the Euclidean proof: split the kernel dyadically, estimate all summands in the $L^2 \to L^p$ topology, and then conclude through real or complex interpolation.
\end{proof}

\subsection{Infinite area quotients} 
\label{infvol}
Only the case of dimension 2 has been explored so far, with the following theorem, which essentially asserts that the bounds are identical to the case of the hyperbolic space, up to possible supbolynomial factors.

\begin{thm}[\cite{AnkerGermainLeger}]
Consider a geometrically finite hyperbolic quotient $M = \Gamma \backslash \mathbb{H}^2$ of the hyperbolic space by a Fuchsian subgroup. For any $\epsilon,N > 0$,
$$
\|P_{\lambda,\delta}\|_{L^2\to L^p}\lesssim_{p,\epsilon}
\lambda^{\gamma(p)+\epsilon}\delta^{\frac12} \quad \mbox{if $\delta > \lambda^{-N}$}.
$$
Furthermore, $\epsilon$ and $N$ can be taken equal to $0$ and $\infty$ respectively if the exponent of convergence of $N$ is $< \frac12$.
\end{thm}

\begin{proof} 
\underline{Decomposition into compact and unbounded parts.}
By a structure theorem for hyperbolic surfaces~\cite{Borthwick}, $M$ can be written as the union of a compact part, hyperbolic ends (identical to hyperbolic cylinders) and cusps. It is easy to see that cusps prevent boundedness of the projectors, and thus we only consider the case when they are absent of the decomposition. We now consider a partition of unity
$$
\chi_{c}(x) + \sum_{j=1}^N \chi_{hf}^j(x) = 1 \qquad \mbox{for all $x \in M$},
$$
where $\chi_c$ is compactly supported, the supports of the $\chi_{hf}^j$ are disjoint, and $M$ reduces to a hyperbolic funnel (halves of hyperbolic cylinders) on the support of each $\chi_{hf}^j$.

At this point, it is convenient to replace $P_{\lambda,\delta}$ by a smoothed-out version of this projector (which is indifferent as far as $L^2 \to L^p$ bounds are concerned), and to express it through the Schr\"odinger group
\begin{equation}
\label{projectorfromschrodinger0}
P_{\lambda,\delta}^\flat = \chi_0 \left( \frac{\Delta + \lambda^2}{\lambda \delta}\right) = \frac{\lambda \delta}{\sqrt{2\pi}} \int \widehat{\chi_0} (\lambda \delta t) e^{i\lambda^2 t}  e^{it \Delta} \,dt,
\end{equation}
where $\chi_0$ is a smooth cutoff function. Combining this formula with the partition of unity defined earlier, we will rely on the decomposition (simplifying slightly the formulas)
$$
P_{\lambda,\delta}^\flat = P_{\lambda,1} \chi_c(x)  P_{\lambda,\delta}^\flat + P_{\lambda,1} \chi_{hf}(x) P_{\lambda,\delta}^\flat.
$$

\medskip
\noindent
\underline{Bounding the compact part.} The crucial ingredient is the local $L^2$ smoothing estimate, which follows from Bourgain-Dyatlov~\cite{BourgainDyatlov}. As explained in Subsection~\ref{localsmoothing} below, this results in the local smoothing bound
$$
\| \chi_c(x) e^{it\Delta} f \|_{L^2_t L^2_x} \lesssim \lambda^{-\frac{1}{2} + \epsilon} \| f \|_{L^2_x}
$$
(for any $\epsilon>0$). Combining this bound with formula~\eqref{projectorfromschrodinger0}, and the Cauchy-Schwarz inequality, we see that
$$
\|  \chi_c(x)  P_{\lambda,\delta}^\flat \|_{L^2} \lesssim \lambda \delta \left\| \chi_c(x) \int \widehat{\chi_0} (\lambda \delta t) e^{i\lambda^2 t} P_{\lambda} e^{it \Delta} \,dt \right\|_{L^2} \lesssim (\lambda \delta)^{\frac{1}{2}} \| \chi_c(x) e^{it\Delta} f \|_{L^2_t L^2_x} \lesssim \lambda^\epsilon \delta^{\frac{1}{2}} \| f \|_{L^2_x}.
$$
Using in addition the Sogge theorem, we can bound
$$
\left\| P_{\lambda,1} \chi_c(x)  P_{\lambda,\delta}^\flat f \right\|_{L^p} \lesssim \lambda^{\gamma(p)} \left\|  \chi_c(x) P_{\lambda,\delta}^\flat f \right\|_{L^2} \lesssim \lambda^{\gamma(p)+\epsilon} \delta^{\frac{1}{2}}.
$$

\medskip
\noindent
\underline{Bounding the hyperbolic part.} To treat the contribution of the hyperbolic ends, the idea is to write 
$$
\chi_{hf}(x) \chi_0 \left( \frac{\Delta + \lambda^2}{\lambda \delta}\right) = \frac{\lambda \delta}{\sqrt{2\pi}} \int \widehat{\chi_0} (\lambda \delta t) e^{i\lambda^2 t} P_{\lambda} \chi_{hf}(x) e^{it \Delta} \,dt.
$$
Then, the function $u = \chi_{hf}^j(x) e^{it \Delta}$ is supported on the $j$-th funnel where it satisfies the Schr\"odinger equation
$$
i \partial_t u - \Delta u = [\Delta, \chi_{hf}^j ] e^{it\Delta} f.
$$
The source term (on the right) is controlled by the local smoothing estimates; in order to control $u$, there remains to derive sufficiently sharp Strichartz, smoothing, and dispersive estimates on the funnel.
\end{proof}

\subsection{Finite area quotients} In the case where $\Gamma \backslash \mathbb{H}^2$ has finite area, Iwaniec and Sarnak~\cite{IwaniecSarnak1995}, see also the review~\cite{Sarnak}, conjectured the following: any eigenfunction $\phi_{\lambda_j}$ associated to the eigenvalue $\lambda_j$ satisfies, for $K$ compact and any $\epsilon>0$
$$
\| \varphi_{\lambda_j} \|_{L^\infty(K)} \lesssim \lambda_j^\epsilon \| \phi_\lambda \|_{L^2}
$$
(which reduces to $\| \varphi_\lambda \|_{L^\infty(X)} \lesssim \lambda^\epsilon$ on a compact surface, for a normalized eigenfunction). 

This conjecture should be regarded as very hard, since, for $p=\infty$ and in the particular case of the modular surface, it implies the Lindel\"of hypothesis~\cite{sarnak0}. We refer to~\cite{ButtcaneKhan2017,Humphries2018,HumphriesKhan2022} for progress on this problem in the case of arithmetic manifolds.

Note that a markedly different behavior can occur in dimension $d = 3$, as evidenced by the result in~\cite{RudnickSarnak}.

\section{Classical estimates implying spectral projector bounds}

In this section, we examine how estimates for the spectrum of the Schr\"odinger operator (Weyl law), or the decay of the Schr\"odinger group (smoothing estimates, dispersive estimates, Strichartz estimates) imply bounds for the spectral projectors $P_{\lambda,\delta}$ in various regimes.

\subsection{Weyl law} \label{sectionweyl}
If $M$ is compact, let $(\lambda_j^2)$ denote the eigenvalues and $(\varphi_j)$ the normalized eigenfunctions of the Laplace-Beltrami operator. The counting function enjoys the following asymptotics~\cite{Ivrii,SoggeHangzhou}
$$
N(\lambda) = \# \{ \lambda_j^2 \; \mbox{eigenvalue of $-\Delta$ with}\; \lambda_j < \lambda \} = C \lambda^d + O(\lambda^{d-1})
$$
(for a constant $C$ corresponding to the volume of the unit cotangent bundle). 
This estimate is sharp for spheres and Zoll manifolds, but the error term is expected to be smaller than $O(\lambda^{d-1})$ for "most" manifolds. Known improvements are as follows: $o(\lambda^{d-1})$ if the set of closed geodesics has measure zero~\cite{DG}, $O(\lambda^{d-1} (\log \lambda)^{-1})$ if $M$ has negative curvature~\cite{Berard}, and  $O(\lambda^{d-1-\delta})$ (for some $\delta>0$ depending on $M$) if $M$ is integrable~\cite{CdV}.

Local Weyl laws refer to (pointwise) estimates on the function
$$
N_x(\lambda) = \sum_{\lambda_j < \lambda} |\varphi_j(x)|^2, \qquad \mbox{with $x \in M$ and $\lambda >0$}.
$$
For any compact manifold~\cite{Hormander,SZ},
\begin{equation}
\label{lWb}
N_x(\lambda) = C \lambda^d + O(\lambda^{d-1}) \qquad \mbox{(uniformly in $x$)}.
\end{equation}
To see the connection to operator norms of $P_{\lambda,\delta}$, observe that the kernel of $P_{\lambda,\delta}$ is given by $K_{\lambda,\delta}(x,y) = \sum_{\lambda-\delta < \lambda_j < \lambda + \delta} \varphi_j(x) \varphi_j(y)$.
By the Cauchy-Schwarz inequality, 
$$
|K_{\lambda,\delta}(x,y)| \leq \sqrt{N_x(\lambda+\delta)-N_x(\lambda-\delta)}\sqrt{N_y(\lambda+\delta)-N_y(\lambda-\delta)}.
$$
Using in addition the bound~\eqref{lWb} gives the theorem of Sogge~\eqref{universalsogge} for $p=\infty$, namely 
$$
\| P_{\lambda,1} \|_{L^1 \to L^\infty} \lesssim \lambda^{d-1}.
$$
 Any improvement of the error term in~\eqref{lWb} leads to an improved estimate for $\| P_{\lambda,\delta} \|_{L^1 \to L^\infty}$.

\subsection{Local smoothing estimates} 
\label{localsmoothing}
Assume that the resolvent $R(\lambda) = (\Delta - \lambda^2 - i0)^{-1}$ enjoys the local smoothing bound
$$
\| \chi R(\lambda) \chi \|_{L^2 \to L^2} \lesssim \lambda^{-1+\epsilon};
$$
here, $\epsilon \geq 0$ and $\chi$ is a localized function which can be compactly supported or decay at a given rate depending on the context. It follows from Kato's theorem that the Schr\"odinger group also exhibits local smoothing:
\begin{equation}
\label{localsmoothing}
\left\| \chi P_{\lambda,1} e^{-it\Delta} f \right\|_{L^2_t L^2(M)} \lesssim \lambda^{-\frac{1}{2}+\epsilon} \| f\|_{L^2}.
\end{equation}
Bounds of this type were proved for various infinite volume manifolds, in particular asymptotically Euclidean or hyperbolic, basic references include~\cite{ConstantinSaut,Doi1,Doi2,KatoYajima}.

It is then possible to follow ideas similar to those presented in Subsection~\ref{infvol}.

\subsection{Dispersive estimates} Assume that dispersive estimates hold on $M$ with the same decay rate as on Euclidean space 
\begin{equation}
\label{dispersive}
\| e^{it\Delta} \|_{L^{p'} \to L^p} \lesssim |t|^{\frac{d}{p} - \frac{d}{2}}, \qquad p\geq 2
\end{equation}
(see the review~\cite{Schlag} and also~\cite{SchlagSofferStaubach1,SchlagSofferStaubach2}). Expressing again the smoothed-out operator $P_{\lambda,\delta}^\flat$ in terms of the Schr\"odinger group, it can be decomposed into
$$
P_{\lambda,\delta}^\flat = \frac{\lambda \delta}{\sqrt{2\pi}} \int \widehat{\chi} (\lambda \delta t) e^{i\lambda^2 t} e^{it\Delta} \,dt
\sim \frac{\lambda \delta}{\sqrt{2\pi}} \left[ \int_{|t| \lesssim \frac{1}{\lambda}} +  \int_{\frac{1}{\lambda} \lesssim |t| \lesssim \frac{1}{\lambda \delta}} \right] e^{i\lambda^2 t} e^{it\Delta} \,dt = P_{small} + P_{large}
$$
(here, integration bounds such as $|t| \lesssim \frac{1}{\lambda}$ should be understood in terms of smooth cutoffs). 

Adjusting the cutoff function, the operator $P_{small}$ can be written as $\delta P_{\lambda,1}^\flat$, to which Sogge's theorem applies: $\|P_{small}\|_{L^{p'} \to L^p} \lesssim \lambda^{d-1-\frac{2d}{p}} \delta
$ if $p \geq p_{ST}$. Turning to $P_{large}$, it can be bounded through the dispersive estimates~\eqref{dispersive} to obtain, if $p \geq p^* = \frac{2d}{d-2}$,
$$
\left\| \int_{\frac{1}{\lambda} \lesssim |t| \lesssim \frac{1}{\lambda \delta}}  \lambda \delta \widehat{\chi} (\lambda \delta t) e^{i\lambda^2 t} e^{it\Delta} \,dt \right\|_{L^{p'} \to L^p} \lesssim \lambda \delta \int_{\frac{1}{\lambda} \lesssim |t| \lesssim \frac{1}{\lambda \delta}} |t|^{\frac{d}{p} - \frac{d}{2}} \,dt \lesssim \lambda^{\frac{d}{2} - \frac{d}{p}} \delta \lesssim \lambda^{d-1-\frac{2d}{p}} \delta.
$$

Overall, we found that, if the dispersive estimates~\eqref{dispersive} hold, the $L^2 \to L^p$ norm of $P_{\lambda,\delta}$ enjoys the optimal bound $\lambda^{\frac{d-1}{2}-\frac{d}{p}} \delta^{\frac 12}$ as soon as $p \geq p^* = \frac{2d}{d-2}$.

\subsection{Strichartz estimates} Strichartz estimates on $\mathbb{R}^d$ are given by
$$
\| e^{it\Delta} f \|_{L^p_t L^q_x} \lesssim \| f \|_{L^2}, \quad \mbox{if $p,q \geq 2$, \;\; $\frac{2}{p} + \frac{d}{q} = \frac{d}{2}$}.
$$
This estimate was first proved~\cite{GV1,GV2,Yajima} for $p>2$, and the endpoint result for $p = p^* = \frac{2d}{d-2}$
$$
\|  e^{it\Delta} f \|_{L^2_t L^{p^*}_x} \lesssim \| f \|_{L^2}
$$
(which will be of particular relevance for us) was then obtained by Keel and Tao~\cite{KT}.

Strichartz estimates were subsequently extended to perturbations of the Euclidean space: either short range perturbations yielding a global in time estimate~\cite{ST,RZ}, or long range perturbations yielding local in time estimates~\cite{Burq,BT}

Assuming that this endpoint estimate holds on $M$, we can apply it together with the Cauchy-Schwarz inequality to the mollified spectral projector expressed through the Schr\"odinger group~\eqref{projectorfromschrodinger0} and obtain
$$
\| P_{\lambda,\delta}^\flat f \|_{L^{p^*}} \lesssim \lambda \delta (\lambda \delta)^{-\frac 12} \| e^{it\Delta} f \|_{L^2([-\frac{1}{\lambda \delta},\frac{1}{\lambda \delta}], L^{p^*}(M))} \lesssim \lambda^{\frac 12} \delta^{\frac 12} \| f \|_{L^2};
$$
(neglecting tails of $\widehat{\chi_0}$).
This shows that the endpoint Strichartz estimate up to time $\sim \frac{1}{\lambda \delta}$ implies the optimal $L^2 \to L^{p^*}$ estimate for $P_{\lambda,\delta}$.

\section{$L^2 \to L^p$ projector bounds beyond complete smooth Riemannian manifolds}

While the present review focuses on bounds for spectral projectors on smooth complete Riemannian manifolds, many interesting variants can be considered. In the list below, we try to summarize all  generalizations which have been investigated, with the relevant references.

\begin{itemize}
\item Non-smooth metric~\cite{KST1,KST2}

\medskip
\item Manifold with boundary~\cite{SoggeBoundary,SmSo}
\medskip

\item Sub-Riemannian manifolds~\cite{Muller, BBG}
\medskip

\item Pseudo-differential operators \cite{KTZ}
\medskip

\item Schr\"odinger operators~\cite{Karadzhov,KoTa,JLR,BHSS,HS}. In the asymptotically Euclidean case, a related question is the boundedness of wave operators~\cite{Schlag}, which can be combined with the Stein-Tomas theorem to prove a natural analog.
\medskip

\item Restriction of eigenfunctions to lower-dimensional submanifolds~\cite{BGT,BR,CS,Tacy}
\medskip

\item Finite graphs (size $N$), for which one should substitute the infinite size limit $N \to \infty$ to the high frequency limit $\lambda \to \infty$ which is of interest for manifolds. Nearly optimal pointwise bounds for eigenvectors could be proved for random graphs~\cite{BaHuYa}. 
\end{itemize}

\section{Some interesting questions} We conclude with some interesting questions

\subsection{Geometry of maximizers.} In all known examples, the functions at which the operator norm of $P_{\lambda,\delta}$ is nearly achieved concentrate asymptotically (as $\lambda \to \infty$) either at a point, or along a geodesic. Could other possibilities arise?

\subsection{Infinite volume manifolds}. In all known examples, $\| P_{\lambda,\delta} \|_{L^2 \to L^p} \lesssim \lambda^{\frac{d-1}{2} - \frac{d}{p}} \delta^{\frac 12} + \lambda^{\frac{d-1}{2}\left(\frac{1}{2} - \frac{1}{p} \right)}$, where the first term comes from point focusing, and the second from geodesic focusing on a stable closed geodesic. Is this bound indeed universal? At least for all manifolds with a simple behavior at infinity (e.g. flat, conic, hyperbolic)?

For $p< p_{ST}$, very different behaviors can occur for $\| P_{\lambda,\delta} \|_{L^2 \to L^p}$. One finds that this is $\sim  \lambda^{\frac{d-1}{2}\left(\frac{1}{2} - \frac{1}{p} \right) } \delta^\alpha$, with $\alpha = 0$ if there exists a stable closed geodesic, $\alpha = \frac{d+1}{2} \left( \frac{1}{2} - \frac 1p \right)$ on the Euclidean space, $\alpha =  \frac{d-1}{2} \left( \frac{1}{2} - \frac 1p \right)$ on Euclidean cylinders, and $\alpha = \frac 12$ on the hyperbolic space (these values of $\alpha$ can be interpreted in relation to the stability of geodesics). Are other values of $\alpha$ possible?

\subsection{Equidistribution of eigenfunctions and of the spectrum}. In the compact case, is it possible to establish a relation between equidistribution of eigenfunctions on the manifold, and equidistribution of the spectrum over the reals? For instance, denoting $\eta$ for a quantity related to the scale over which eigenvalues are equidistributed, could there hold an estimate such as $\| \varphi \|_{L^\infty} \lesssim \lambda^{\frac{d-1}{2} - \frac{d}{p}} \eta^{\frac{1}{2}}$ for any eigenfunction $\varphi$ with eigenvalue $\lambda$, and for $p$ sufficiently large? On the one hand, such an estimate seems to fail for surfaces of revolution, whose spectrum can be equidistributed, but which admit exceptional eigenfunctions concentrating at the poles. On the other hand, such an estimate is verified on spheres and rational or irrational tori. An analogy can also be made with random matrices, for which a connection between quantum unique ergodicy and GUE statistics for the spectrum can be established in some cases~\cite{Bourgade}.

\subsection{Ignoring stable closed geodesics} Stable closed geodesics are the obstacle to an improvement over the sphere case for $p$ close to $2$. Would an improvement become possible by localizing away from stable geodesics?

\subsection{Link with the Fourier restriction problem}. The question of estimating $\| P_{\lambda,\delta} \|_{L^2 \to L^p}$ provided a generalization of the Fourier restriction problem to arbitrary Riemannian manifolds, when the density in Fourier space is $L^2$. Is it possible (and interesting) to obtain such a generalization for the analog of $L^p$ density in Fourier space? 

The formulation of this question depends on a choice of a distinguished basis of (generalized) eigenfunctions. If the manifold is of infinite volume and for instance asymptotically Euclidean, such a distinguished basis could be provided by scattering theory. If the manifold is compact, then the eigenspaces of the Laplacian are generically non-degenerate, and this provides naturally an orthonormal basis of $L^2$, which will be denoted $(\varphi_k)$. It is tempting to ask about the $\ell^q(\mathbb{N}) \to L^p(M)$ boundedness properties of $(c_k)_{k \in \mathbb{N}} \mapsto \sum_{|\lambda_k - \lambda| < \delta} c_k \varphi_k$ (denoting $\varphi_k$ for the $k$-th eigenfunction). In the case of the torus, it was noticed by Bourgain and Demeter~\cite{BD1} that the case of general $q$ is an immediate consequence of $q=2$. Does this remain true for other manifolds?

\subsection{Nonsymmetric and nonintegrable compact manifolds} We conclude with this most ambitious question: if the manifold under consideration is neither symmetric, nor integrable, there does not exist, to the best of our knowledge, any result giving more than logarithmic improvements over the universal Sogge estimate. Could that be obtained in a single example?

\subsection*{Acknowledgements} The author is grateful to Tristan L\'eger for his suggestions to improve a preliminary version, and to Paul Bourgade, Isabelle Gallagher, Chris Sogge, Stefan Steinerberger, Daniel Tataru, Laurent Thomann and Maciej Zworski for enlightening conversations during the writing of this review article.

When working on this article, the author was supported by a start-up grant from Imperial College and a Wolfson fellowship

\bibliographystyle{abbrv}
\bibliography{references}

\end{document}